\documentclass[11pt]{article}

\usepackage[T1]{fontenc}
\usepackage[utf8]{inputenc}
\usepackage{amsmath,amssymb,amsthm}
\usepackage{mathtools}
\usepackage[margin=1in]{geometry}
\usepackage{enumitem}

\numberwithin{equation}{section}

\newtheorem{theorem}{Theorem}[section]
\newtheorem{conjecture}[theorem]{Conjecture}
\newtheorem{lemma}[theorem]{Lemma}
\newtheorem{proposition}[theorem]{Proposition}
\newtheorem{remark}[theorem]{Remark}

\newcommand{\Zn}{\mathbb Z_n}
\newcommand{\Ehat}{\widehat E_q}

\title{The Type III realisation conjecture of Kirkland and \v{S}migoc}
\author{Brecht Verbeken$^{1,2}$ and Vincent Ginis$^{1,2,3}$\\[0.4em]
\small $^{1}$Department of Business Technology and Operations, Data Analytics Laboratory,\\
\small Vrije Universiteit Brussel (VUB), Pleinlaan 2, 1050 Brussels, Belgium\\
\small $^{2}$imec-SMIT, Vrije Universiteit Brussel, Pleinlaan 9, 1050 Brussels, Belgium\\
\small $^{3}$School of Engineering and Applied Sciences, Harvard University,\\
\small Cambridge, Massachusetts 02138, USA\\[0.4em]
\texttt{brecht.verbeken@vub.be}\quad
\texttt{vincent.ginis@vub.be}}
\date{}

\begin{document}
\maketitle

\begin{abstract}
Kirkland and \v{S}migoc constructed a family of stochastic matrices realising
the Type III boundary polynomials in the Karpelevi\v{c} region and conjectured
that, conversely, every stochastic realisation of such a polynomial must come
from their construction. We prove this conjecture for the full nonzero
parameter range $0<\alpha\le1$, for genuine Type III reduced Ito polynomials
of order $n$
\[
        f_\alpha(x)=x^y\bigl(x^q-(1-\alpha)\bigr)^d-\alpha^d,
        \qquad n=qd+y.
\]
For $0<\alpha<1$, the proof first reduces every realisation to a two-shift
cyclic normal form using the Dmitriev--Dynkin boundary theorem. The remaining
argument is finite and combinatorial: Coates' coefficient formula and the
equality case of a weighted Tur\'an theorem force the $q$-cycles associated
with the backward edges to split into $d$ complete multipartite classes of
equal total weight. A circular-arc telescoping argument then converts this
additive equality into the product condition required by Kirkland and
\v{S}migoc. The endpoint $\alpha=1$ is treated separately. We also explain
why the closed endpoint $\alpha=0$ is degenerate: the literal extension to
this endpoint fails, because reducible realisations with closed $q$-cycles and
transient states need not contain the global $n$-cycle.
\end{abstract}



\section{Introduction}

Let $\Theta_n$ denote the set of all eigenvalues of all $n\times n$ stochastic
matrices. Karpelevi\v{c}'s theorem describes $\Theta_n$ \cite{Karpelevic},
and Ito's reformulation expresses the boundary arcs of $\Theta_n$ through
one-parameter families of polynomials associated with Farey-neighbour data
\cite{Ito}. Following the terminology used by Johnson--Paparella
\cite{JohnsonPaparella} and by Kirkland--\v{S}migoc \cite{KirklandSmigoc},
the resulting reduced Ito polynomials fall into Types 0, I, II, and III.
Write $\beta=1-\alpha$, with $0\le\alpha\le1$. A convenient list of reduced
forms is
\[
\begin{aligned}
\text{Type 0:}\quad
&f_\alpha(t)=(t-\beta)^n-\alpha^n,\\
\text{Type I:}\quad
&f_\alpha(t)=t^n-\beta t^{n-q}-\alpha,
\qquad q>\frac n2,\\
\text{Type II:}\quad
&f_\alpha(t)=(t^q-\beta)^d-\alpha^dt^z,
\qquad n=qd,\ 1\le z\le q-1,\ d\ge2,\\
\text{Type III:}\quad
&f_\alpha(t)=t^y(t^q-\beta)^d-\alpha^d,
\qquad n=qd+y,\ 1\le y\le q-1,\ d\ge2.
\end{aligned}
\]
For this paper, a \emph{genuine Type III reduced Ito polynomial of order
$n$} means a polynomial of the form
\[
        f_\alpha(x)=x^y(x^q-(1-\alpha))^d-\alpha^d,
        \qquad n=qd+y,
\]
arising from the Type III Farey-neighbour case in the Ito--Karpelevi\v{c}
parametrisation with associated open boundary arc of exact order $n$. Thus
the exact-order boundary condition is part of the word genuine below, not an
additional hypothesis. The proof uses this genuineness only to ensure that,
for $0<\alpha<1$, the relevant root lies on the open Karpelevi\v{c} boundary
arc and belongs to $\Theta_n\setminus\Theta_{n-1}$. The remaining finite
argument uses only
\[
        d\ge2,\qquad 1\le y\le q-1,\qquad \gcd(q,n)=1.
\]

Kirkland and \v{S}migoc initiated a structural realisation theory for these
boundary polynomials \cite{KirklandSmigoc}. They characterised all stochastic
matrices whose characteristic polynomials are of Type 0 or Type I, and
characterised the sparsest realisations for Type II and Type III.

Let
\[
        f_\alpha(x)=x^y(x^q-\beta)^d-\alpha^d,
        \qquad \beta=1-\alpha,
        \qquad n=qd+y,
\]
where $d\ge2$ and $1\le y\le q-1$. In the genuine Type III situation this
polynomial comes from a Farey pair with denominators $q$ and $n$, and hence
$\gcd(q,n)=1$.

Let $\Zn=\mathbb Z/n\mathbb Z$. On the vertex set $\Zn$, put
\[
        E_n=\{(i,i+1):i\in\Zn\},
        \qquad
        \Ehat=\{(i,i+1-q):i\in\Zn\}.
\]
More generally, write
\[
        T_b=\{(i,i+b):i\in\Zn\}.
\]
Let $C_n$ denote the cyclic permutation matrix with edges $E_n$; then
$C_n^b$ has edge set $T_b$.

For $i,j\in\Zn$, define
\[
        d_n(i,j)=\min\{[i-j]_n,[j-i]_n\},
\]
where $[r]_n$ is the least nonnegative residue of $r$ modulo $n$.

Kirkland and \v{S}migoc proved that any stochastic matrix satisfying the
graph-and-weight conditions recalled below has characteristic polynomial
$f_\alpha$ \cite[Proposition 7.1]{KirklandSmigoc}. They conjectured the
following converse to their Proposition 7.1.

\begin{conjecture}[Kirkland--\v{S}migoc, Conjecture 7.1]\label{conj:KS}
Let $A$ be a stochastic matrix of order $n$ with Type III characteristic
polynomial
\[
        f_\alpha(x)=x^y\bigl(x^q-(1-\alpha)\bigr)^d-\alpha^d,
        \qquad n=qd+y,
        \qquad 1\le y\le q-1,
        \qquad d\ge2.
\]
Then $A$ is permutation similar to a matrix whose directed graph $\Gamma$
satisfies the conditions of Proposition 7.1 in Kirkland and \v{S}migoc.
\end{conjecture}

We recall these conditions explicitly. The digraph has edge set
\[
        E(\Gamma)=E_n\cup\bigcup_{r=1}^d B_r,
        \qquad B_r\subseteq\Ehat,
\]
where
\[
        (i,i+1-q),(j,j+1-q)\in B_r
        \Longrightarrow d_n(i,j)<q, \tag{P1}
\]
\[
        (i,i+1-q)\in B_r,\ (j,j+1-q)\in B_s,\ r\ne s
        \Longrightarrow d_n(i,j)\ge q. \tag{P2}
\]
Moreover, if $u_i$ denotes the weight of the forward edge $(i,i+1)$, then
\[
        \prod_{\{i:(i,i+1-q)\in B_r\}}u_i=\alpha,
        \qquad r=1,\ldots,d. \tag{P3}
\]

The main result of this paper is the converse for $0<\alpha\le1$. The
non-degenerate case is $0<\alpha<1$, and the endpoint $\alpha=1$ is dealt with
separately.

\begin{theorem}[Main theorem]\label{thm:main}
Let
\[
        f_\alpha(x)=x^y\bigl(x^q-(1-\alpha)\bigr)^d-\alpha^d,
        \qquad n=qd+y,
\]
be a genuine Type III reduced Ito polynomial of order $n$, so that $d\ge2$,
$1\le y\le q-1$, and $\gcd(q,n)=1$. If $0<\alpha\le1$ and $A$ is an
$n\times n$ stochastic matrix with characteristic polynomial $f_\alpha$, then
$A$ is permutation similar to a matrix whose directed graph satisfies
(P1)--(P3).
\end{theorem}

Thus Conjecture \ref{conj:KS} is true for $0<\alpha\le1$. The endpoint
$\alpha=0$ is exceptional; see Section \ref{sec:endpoints}. The obstruction is
reducibility: at $\alpha=0$ the polynomial can be realised by $d$ closed
$q$-cycles together with $y$ transient states, so the characteristic polynomial
no longer forces the global $n$-cycle present on the open Type III arc.
Consequently, the statement should either be read as an open-arc assertion, or
supplemented at $\alpha=0$ by an additional exact-order/irreducibility
hypothesis.

\section{Preliminaries}

\subsection{Weighted digraphs and Coates' formula}

A weighted digraph $\Gamma=(V,E,w)$ has positive weights on its edges. Its
adjacency matrix $A=(a_{ij})$ is defined by $a_{ij}=w(i,j)$ if $(i,j)\in E$
and $a_{ij}=0$ otherwise. Directed cycles are allowed to have length one; thus
a loop is regarded as a directed $1$-cycle. A linear digraph is a
vertex-disjoint union of directed cycles. If $L$ is a linear digraph, let
$c(L)$ denote its number of cycles and let
\[
        \pi(L)=\prod_{e\in E(L)}w(e)
\]
be its weight.

We use the following coefficient form of Coates' determinant formula.

\begin{lemma}[Coates linear-digraph formula]
Let $A$ be the adjacency matrix of a weighted digraph $\Gamma$ on $n$
vertices, and write
\[
        \det(xI_n-A)=x^n+k_1x^{n-1}+\cdots+k_n.
\]
If $\mathcal L_m$ denotes the set of linear digraphs of $\Gamma$ whose total
number of vertices is $m$, then
\[
        k_m=\sum_{L\in\mathcal L_m}(-1)^{c(L)}\pi(L).
\]
\end{lemma}

\subsection{The exact Type III boundary input and the two-shift normal form}

We isolate the only part of the argument that uses the global theory of the
Karpelevi\v{c} region.

Let
\[
        f_\alpha(x)=x^y(x^q-\beta)^d-\alpha^d,
        \qquad \beta=1-\alpha,
        \qquad n=qd+y,
\]
where $d\ge2$, $1\le y\le q-1$, and the parameters come from a Type III
Farey-neighbour arc in the genuine sense fixed above. We shall use this
genuineness in the following sense: for $0<\alpha<1$, one of the roots
$\lambda$ of $f_\alpha$ lies on the open Karpelevi\v{c} boundary arc
associated with the Type III Farey-neighbour data and satisfies
\[
        \lambda\in\Theta_n\setminus\Theta_{n-1}.
\]
This is the only analytic input from the Karpelevi\v{c}--Ito parametrisation.
The remaining argument is finite and combinatorial.

Kirkland and \v{S}migoc recall Ito's description of the boundary arcs and the
reduced Ito polynomial types; in their Type III case the reduced polynomial
has the form
\[
        f_\alpha(t)=t^y(t^q-\beta)^d-\alpha^d,
        \qquad n=s=qd+y,
        \qquad 1\le y\le q-1,
        \qquad d\ge2.
\]
They also explain that the $n$-dimensional cases are the relevant ones for
excluding $\Theta_{n-1}$.

We shall use the following two results from Kirkland--\v{S}migoc.

First, their Theorem 3.1, attributed to Dmitriev and Dynkin, says that if
\[
        \lambda\in\Theta_n\setminus\Theta_{n-1},
        \qquad
        \frac{2\pi k}{n}\le \arg\lambda\le\frac{2\pi(k+1)}{n},
\]
then every $n\times n$ stochastic matrix having $\lambda$ as an eigenvalue is
permutation similar to a matrix whose nonzero entries lie only in two
consecutive cyclic positions. In the notation
$T_b=\{(i,i+b):i\in\Zn\}$, this says
\[
        E(\Gamma)\subseteq T_a\cup T_{a+1}
\]
for some $a\in\Zn$. This is just their one-based statement translated into
zero-based cyclic-shift notation.

Second, their Proposition 3.1 says that if
$\lambda\in\Theta_n\setminus\Theta_{n-1}$ is a root of a reduced Ito
polynomial with parameters $q,s,d$, then the digraph of any stochastic
realisation contains at least one $s$-cycle, at least one $q$-cycle, and no
cycles of length different from $s$ and $kq$, $1\le k\le d$. For the Type III
case used here, $s=n$. Hence the digraph contains an $n$-cycle and a
$q$-cycle, and every directed cycle has length either $n$ or $kq$,
$1\le k\le d$.

Now we derive the normal form needed below.

\begin{lemma}[Two-shift normal form]\label{lem:two-shift}
Let
\[
        f_\alpha(x)=x^y(x^q-\beta)^d-\alpha^d,
        \qquad \beta=1-\alpha,
        \qquad n=qd+y,
\]
be a genuine Type III reduced Ito polynomial of order $n$, with
\[
        d\ge2,
        \qquad 1\le y\le q-1,
        \qquad \gcd(q,n)=1.
\]
Assume $0<\alpha<1$, and let $A$ be an $n\times n$ stochastic matrix with
characteristic polynomial $f_\alpha$. Then, after permutation similarity, the
weighted digraph $\Gamma$ of $A$ satisfies
\[
        E_n\subseteq E(\Gamma)\subseteq E_n\cup\Ehat,
\]
where
\[
        E_n=\{(i,i+1):i\in\Zn\},
        \qquad
        \Ehat=\{(i,i+1-q):i\in\Zn\}.
\]
Equivalently,
\[
        A=DC_n+(I_n-D)C_n^{\,1-q},
        \qquad D=\operatorname{diag}(u_i)_{i\in\Zn},
        \qquad
        0<u_i\le1 \qquad (i\in\Zn).
\]
Moreover every directed cycle of $\Gamma$ has length either $n$ or $kq$ for
some $1\le k\le d$.
\end{lemma}

\begin{proof}
Choose the boundary root $\lambda$ of $f_\alpha$ lying on the open Type III
Karpelevi\v{c} arc.
Since
\[
        f_\alpha(0)=-\alpha^d\ne0
\]
for $0<\alpha<1$, this root is nonzero and $\arg\lambda$ is defined. By
genuineness of the Type III data,
\[
        \lambda\in\Theta_n\setminus\Theta_{n-1}.
\]
Choose $a\in\Zn$ such that
\[
        \frac{2\pi a}{n}\le \arg\lambda\le\frac{2\pi(a+1)}{n}.
\]
By the Dmitriev--Dynkin theorem as stated by Kirkland--\v{S}migoc, after
permutation similarity,
\[
        E(\Gamma)\subseteq T_a\cup T_{a+1}.
\]

By Kirkland--\v{S}migoc Proposition 3.1, in the Type III case $s=n$, so
$\Gamma$ contains an $n$-cycle and a $q$-cycle, and every cycle length is
either $n$ or $kq$, $1\le k\le d$.

Let $C$ be an $n$-cycle of $\Gamma$. Suppose $C$ uses $b$ edges of type
$T_{a+1}$ and $n-b$ edges of type $T_a$. The total displacement of $C$ modulo
$n$ is
\[
        (n-b)a+b(a+1)=na+b\equiv b\pmod n.
\]
Since $C$ is closed, $b\equiv0\pmod n$. Hence $b=0$ or $b=n$. Thus the
$n$-cycle uses only one of the two allowed shifts. Denote that shift by $h$.
Since the cycle has length $n$, the step $h$ has order $n$ in $\Zn$, and the
cycle contains every edge of $T_h$.

Relabel the vertices by the affine permutation $i\mapsto h^{-1}i$. Then
$T_h$ becomes $T_1=E_n$. The other allowed shift becomes some fixed shift
$T_r$, with $r\ne1$. Therefore
\[
        E_n\subseteq E(\Gamma)\subseteq E_n\cup T_r.
\]

It remains to identify $r$. Since $\Gamma$ contains a $q$-cycle and $q<n$,
some edge of $T_r$ must occur. Fix such an edge $(i,i+r)$. Together with the
forward path in $E_n$ from $i+r$ back to $i$, this edge forms a directed cycle
of length
\[
        \ell=1+[-r]_n,
\]
where $[-r]_n$ is the least nonnegative residue of $-r$ modulo $n$. Because
$r\ne1$, this cycle has length $\ell<n$.
By the cycle-length restriction, this cycle has length
\[
        \ell=\kappa q
\]
for some $1\le \kappa\le d$.

We claim that $\kappa=1$. Suppose instead that $\kappa\ge2$. Let $C_q$ be a
$q$-cycle of $\Gamma$, and let $b$ be the number of $T_r$-edges in $C_q$.
Since $q<n$, this cycle cannot consist only of forward edges, so
$1\le b\le q$. Its displacement modulo $n$ is
\[
        (q-b)\cdot1+br=q+b(r-1).
\]
Because $\ell=1+[-r]_n$, we have $r-1\equiv-\ell\pmod n$. Hence
\[
        q+b(r-1)\equiv q-b\ell=q(1-b\kappa)\pmod n.
\]
The cycle is closed, so
\[
        n\mid q(1-b\kappa).
\]
Since $\gcd(q,n)=1$,
\[
        n\mid 1-b\kappa.
\]
Equivalently,
\[
        n\mid b\kappa-1.
\]
But
\[
        1\le b\kappa-1\le qd-1=n-y-1<n,
\]
because $1\le b\le q$, $2\le\kappa\le d$, and $y\ge1$. This is impossible.
Hence $\kappa=1$, so
\[
        \ell=q.
\]
Therefore
\[
        1+[-r]_n=q,
\]
i.e.
\[
        [-r]_n=q-1,
        \qquad\text{so}\qquad
        r\equiv1-q\pmod n.
\]
Thus $T_r=\Ehat$, and
\[
        E_n\subseteq E(\Gamma)\subseteq E_n\cup\Ehat.
\]

Finally, since $A$ is stochastic and each row has support contained in the
two positions $i+1$ and $i+1-q$, write
\[
        u_i=a_{i,i+1}.
\]
The inclusion $E_n\subseteq E(\Gamma)$ gives $u_i>0$, while stochasticity
gives
\[
        a_{i,i+1-q}=1-u_i\ge0.
\]
Hence $0<u_i\le1$, and
\[
        A=DC_n+(I_n-D)C_n^{\,1-q},
        \qquad D=\operatorname{diag}(u_i).
\]
The cycle-length assertion is exactly the Type III specialization of
Kirkland--\v{S}migoc Proposition 3.1.
\end{proof}

\begin{remark}
Lemma \ref{lem:two-shift} is the only point where the global boundary theory
of Karpelevi\v{c} and Dmitriev--Dynkin enters the proof. Everything after it
is a finite weighted-digraph argument.
\end{remark}

\subsection{Weighted Tur\'an equality}

We need the equality case of the weighted Tur\'an theorem. The inequality
follows from the Motzkin--Straus theorem; we include the equality argument.

\begin{lemma}[Weighted Tur\'an theorem]\label{lem:turan}
Let $G=(V,E)$ be a finite graph with positive vertex weights $w_v>0$.
Let $\omega(G)$ be the clique number of $G$, and put
\[
        W=\sum_{v\in V} w_v,
        \qquad
        M(G,w)=\sum_{\{u,v\}\in E} w_u w_v.
\]
If $\omega(G)\le d$, then
\[
        M(G,w)\le \frac{d-1}{2d}W^2.
\]
Equality holds if and only if $V$ admits a partition
\[
        V=V_1\sqcup\cdots\sqcup V_d
\]
such that $G$ is the complete $d$-partite graph with parts
$V_1,\ldots,V_d$, and
\[
        \sum_{v\in V_1}w_v=\cdots=\sum_{v\in V_d}w_v=\frac Wd.
\]
\end{lemma}

\begin{proof}
Set $p_v=w_v/W$. Then $p_v>0$ and $\sum_v p_v=1$. Write
\[
        F(p)=\sum_{\{u,v\}\in E}p_up_v.
\]
By the Motzkin--Straus theorem \cite{MotzkinStraus},
\[
        F(p)\le \frac12\left(1-\frac1{\omega(G)}\right).
\]
Since $\omega(G)\le d$, this gives
\[
        F(p)\le \frac12\left(1-\frac1d\right),
\]
and multiplying by $W^2$ proves the inequality.

Suppose equality holds. Put $m=\omega(G)$. Equality in the two preceding
inequalities forces $m=d$, and $p$ is a maximizer in the Motzkin--Straus
problem for $G$.

For $v\in V$, define
\[
        N_v=\sum_{\{v,z\}\in E}p_z.
\]
Since all $p_v$ are positive, $p$ lies in the relative interior of the
probability simplex. Hence, for any $u,v\in V$, the variation
$p+t(e_u-e_v)$ is feasible for all sufficiently small $t$, and maximality of
$p$ gives
\[
        N_u=N_v.
\]
Thus $N_v=\gamma$ is independent of $v$. Since
\[
        \sum_{v\in V}p_vN_v
        =2\sum_{\{u,v\}\in E}p_up_v
        =2F(p),
\]
we have
\[
        \gamma=2F(p)=1-\frac1m.
\]

Let $C=\{c_1,\ldots,c_m\}$ be a clique of size $m$. Then
\[
        m\gamma
        =\sum_{r=1}^m N_{c_r}
        =\sum_{v\in V}p_v\,|\{r:\{v,c_r\}\in E\}|.
\]
Every vertex is adjacent to at most $m-1$ vertices of $C$, since adjacency to
all of $C$ would create a clique of size $m+1$. Therefore
\[
        m\gamma\le (m-1)\sum_{v\in V}p_v=m-1.
\]
But $\gamma=1-1/m$, so equality holds. Since every $p_v$ is positive, each
vertex is adjacent to exactly $m-1$ vertices of $C$.

Define
\[
        V_r=\{v\in V:\{v,c_r\}\notin E\},\qquad r=1,\ldots,m.
\]
These sets form a partition of $V$, and $c_r\in V_r$. Put
\[
        P_r=\sum_{v\in V_r}p_v.
\]
Because $c_r$ is adjacent precisely to the vertices outside $V_r$,
\[
        N_{c_r}=1-P_r.
\]
Since $N_{c_r}=1-1/m$, we get
\[
        P_r=\frac1m,\qquad r=1,\ldots,m.
\]

There are no edges inside any $V_r$. Indeed, if distinct $x,y\in V_r$ were
adjacent, then
\[
        \{x,y\}\cup (C\setminus\{c_r\})
\]
would be a clique of size $m+1$, a contradiction.

Finally, every edge between different parts is present. Let $x\in V_r$ and
$y\in V_s$ with $r\ne s$. If $\{x,y\}\notin E$, then, since there are no
edges inside $V_r$ and $p_y>0$,
\[
        N_x<1-P_r=1-\frac1m,
\]
contradicting $N_x=\gamma=1-1/m$. Hence all cross-edges are present.

Thus $G$ is complete $m$-partite with part masses $P_r=1/m$. Since $m=d$,
this is exactly the stated equality structure, after multiplying the part
masses by $W$.

Conversely, if $G$ is complete $d$-partite with part weights
\[
        W_r=\sum_{v\in V_r}w_v=\frac Wd,
\]
then
\[
        \begin{aligned}
        M(G,w)
        &=\sum_{1\le r<s\le d}W_rW_s\\
        &=\frac12\left(W^2-\sum_{r=1}^d W_r^2\right)\\
        &=\frac12\left(W^2-d\frac{W^2}{d^2}\right)\\
        &=\frac{d-1}{2d}W^2.
        \end{aligned}
\]
This proves the equality characterization.
\end{proof}

\section{The proof for $0<\alpha<1$}

Throughout this section assume
\[
        0<\alpha<1,\qquad \beta=1-\alpha,\qquad n=qd+y,
        \qquad 1\le y\le q-1,\qquad d\ge2,\qquad \gcd(q,n)=1.
\]
Let $A$ be an $n\times n$ stochastic matrix with characteristic polynomial
\[
        f_\alpha(x)=x^y(x^q-\beta)^d-\alpha^d.
\]
Let $\Gamma$ be the weighted digraph of $A$. By Lemma \ref{lem:two-shift},
after replacing $A$ by a permutation-similar matrix and $\Gamma$ by the
corresponding relabelled weighted digraph, we may assume
\[
        A=DC_n+(I_n-D)C_n^{\,1-q},
        \qquad D=\operatorname{diag}(u_i)_{i\in\Zn},
        \qquad 0<u_i\le1.
\]
Let
\[
        S=\{i\in\Zn:u_i<1\}.
\]
Thus $i\in S$ exactly when the backward edge $(i,i+1-q)$ is present, and then
that edge has weight $1-u_i$.

For $i\in S$, define the cyclic interval
\[
        Q_i=\{i-q+1,i-q+2,\ldots,i\}\subseteq\Zn.
\]
This is the vertex set of the directed $q$-cycle
\[
        i\to i+1-q\to i+2-q\to\cdots\to i-1\to i.
\]
Its weight is
\[
        w_i=(1-u_i)\prod_{r=i-q+1}^{i-1}u_r.
\]
All indices are taken modulo $n$.

\begin{lemma}\label{lem:qcycles}
The $q$-cycles of $\Gamma$ are exactly the cycles $Q_i$, $i\in S$.
\end{lemma}

\begin{proof}
Let $C$ be a directed cycle of length $q$ in the two-shift graph
$E_n\cup\Ehat$. Suppose $C$ uses $b$ backward edges. A forward edge has cyclic
displacement $1$, and a backward edge has cyclic displacement $1-q$. Hence the
total displacement of $C$ is
\[
        (q-b)\cdot1+b(1-q)=q(1-b).
\]
Since $C$ is closed modulo $n$, we have
\[
        n\mid q(1-b).
\]
Because $\gcd(q,n)=1$, this gives $n\mid1-b$. But $0\le b\le q<n$, so
$b=1$. Thus every $q$-cycle uses exactly one backward edge. If the backward
edge is $(i,i+1-q)$, then $i\in S$. The remaining $q-1$ edges are forward
edges, so they are forced to be
\[
        i+1-q\to i+2-q\to\cdots\to i-1\to i.
\]
Hence the cycle has vertex set
\[
        Q_i=\{i-q+1,i-q+2,\ldots,i\}.
\]
Conversely, for every $i\in S$, the backward edge $(i,i+1-q)$ is present, and
all forward edges are present because $E_n\subseteq E(\Gamma)$. Hence $Q_i$ is
indeed a directed $q$-cycle.
\end{proof}

\begin{lemma}\label{lem:turan-step}
Let $G$ be the graph on vertex set $S$ defined by
\[
        \{i,j\}\in E(G)\Longleftrightarrow Q_i\cap Q_j=\varnothing.
\]
Equivalently, since $n=qd+y\ge2q+1$,
\begin{equation}
        \{i,j\}\in E(G) \Longleftrightarrow d_n(i,j)\ge q. \label{eq:dist}
\end{equation}
The graph $G$ is a complete $d$-partite graph
\[
        S=S_1\sqcup\cdots\sqcup S_d
\]
with
\[
        \sum_{i\in S_r}w_i=\beta,\qquad r=1,\ldots,d.
\]
Consequently,
\[
        i,j\in S_r \Longrightarrow d_n(i,j)<q,
\]
\[
        i\in S_r,\ j\in S_s,\ r\ne s \Longrightarrow d_n(i,j)\ge q.
\]
\end{lemma}

\begin{proof}
The equivalence \eqref{eq:dist} follows because the $Q_i$ are cyclic
intervals of $q$ consecutive vertices and $n\ge2q+1$: two such intervals
cannot overlap in both cyclic directions, and they intersect exactly when the
circular distance between their right endpoints is at most $q-1$.

We first record the two coefficient identities that will be used. Since
\[
        f_\alpha(x)=x^y(x^q-\beta)^d-\alpha^d
        =\sum_{r=0}^d (-1)^r\binom dr\beta^r x^{n-rq}-\alpha^d,
\]
the coefficients of $x^{n-q}$ and $x^{n-2q}$ are respectively
\[
        -d\beta,
        \qquad
        \binom d2\beta^2.
\]

By Lemma \ref{lem:two-shift}, every directed cycle of $\Gamma$ has length
$n$ or $kq$, $1\le k\le d$. Since $q<n$, a linear digraph on exactly $q$
vertices can only be a single $q$-cycle. By Lemma \ref{lem:qcycles}, the
$q$-cycles are precisely the $Q_i$, $i\in S$, with weights $w_i$. Therefore
Coates' formula applied to the coefficient of $x^{n-q}$ gives
\[
        -d\beta=-\sum_{i\in S}w_i,
\]
and hence
\begin{equation}
        \sum_{i\in S}w_i=d\beta. \label{eq:totalweight}
\end{equation}
Each $w_i$ is positive, because $i\in S$ implies $1-u_i>0$, and Lemma
\ref{lem:two-shift} gives $u_r>0$ for every forward edge appearing in the
corresponding $q$-cycle.

Next consider the coefficient of $x^{n-2q}$. Since $2q<n$, a linear digraph
on $2q$ vertices cannot contain an $n$-cycle. By the cycle-length restriction
from Lemma \ref{lem:two-shift}, such a linear digraph is therefore either a
disjoint union of two $q$-cycles or a single $2q$-cycle.

The disjoint unions of two $q$-cycles are exactly the pairs
$Q_i\sqcup Q_j$ with $Q_i\cap Q_j=\varnothing$, equivalently
$\{i,j\}\in E(G)$. Their total contribution to Coates' formula is positive,
since they have two cycle components:
\[
        M=\sum_{\{i,j\}\in E(G)}w_iw_j.
\]
A single $2q$-cycle has one cycle component, so its Coates sign is negative.
Let $R\ge0$ be the total weight of all directed $2q$-cycles in $\Gamma$.
Applying Coates' formula to the coefficient of $x^{n-2q}$ gives
\begin{equation}
        \binom d2\beta^2=M-R. \label{eq:coates2q}
\end{equation}
In particular,
\begin{equation}
        M\ge \binom d2\beta^2. \label{eq:lowerM}
\end{equation}

On the other hand, the clique number of $G$ is at most $d$. Indeed, a clique
of size $d+1$ would give $d+1$ pairwise disjoint $q$-subsets of $\Zn$,
requiring $q(d+1)$ vertices, whereas
\[
        q(d+1)>qd+y=n
\]
because $y<q$.

We may therefore apply Lemma \ref{lem:turan} to the graph $G$ with positive
vertex weights $w_i$. Using \eqref{eq:totalweight}, we obtain
\begin{equation}
        \begin{aligned}
        M
        &\le
        \frac{d-1}{2d}\left(\sum_{i\in S}w_i\right)^2\\
        &=\frac{d-1}{2d}(d\beta)^2\\
        &=\binom d2\beta^2.
        \end{aligned}
        \label{eq:upperM}
\end{equation}
Combining \eqref{eq:lowerM} and \eqref{eq:upperM}, we get
\[
        M=\binom d2\beta^2.
\]
Then \eqref{eq:coates2q} forces $R=0$, and equality holds in the weighted
Tur\'an inequality.

By the equality case in Lemma \ref{lem:turan}, $G$ is a complete
$d$-partite graph
\[
        S=S_1\sqcup\cdots\sqcup S_d
\]
such that every part has total weight
\[
        \sum_{i\in S_r}w_i=\frac1d\sum_{i\in S}w_i=\beta,
        \qquad r=1,\ldots,d.
\]

Finally, the distance conclusions follow from the definition of $G$ and the
equivalence \eqref{eq:dist}. Vertices in the same part of a complete
multipartite graph are nonadjacent, so if $i,j\in S_r$, then
$d_n(i,j)<q$. Vertices in different parts are adjacent, so if $i\in S_r$,
$j\in S_s$, and $r\ne s$, then $d_n(i,j)\ge q$. This proves the lemma.
\end{proof}

Lemma \ref{lem:turan-step} proves the graph-distance part of Kirkland--\v{S}migoc
\cite[Proposition 7.1]{KirklandSmigoc}. It remains to prove the product
condition. This is the only point where circular interval geometry is needed.

\begin{lemma}[Circular-arc telescoping]\label{lem:telescoping}
For each part $S_r$ in Lemma \ref{lem:turan-step},
\[
        \sum_{i\in S_r}w_i=1-\prod_{i\in S_r}u_i.
\]
\end{lemma}

\begin{proof}
Fix $r\in\{1,\ldots,d\}$. By Lemma \ref{lem:turan-step},
\[
        \sum_{i\in S_s} w_i=\beta>0
\]
for every $s$, and hence every part $S_s$ is nonempty. For each $s\ne r$,
choose one index
\[
        j_s\in S_s.
\]
Since vertices belonging to different parts of the complete $d$-partite graph
$G$ are adjacent, Lemma \ref{lem:turan-step} gives
\[
        d_n(j_s,j_t)\ge q\qquad (s\ne t),
\]
and therefore the intervals $Q_{j_s}$, $s\ne r$, are pairwise disjoint.
Moreover, again by Lemma \ref{lem:turan-step}, if $i\in S_r$ and $s\ne r$,
then
\[
        d_n(i,j_s)\ge q,
\]
so $Q_i\cap Q_{j_s}=\varnothing$. Thus every interval $Q_i$, $i\in S_r$, is
contained in the complement of the union
\[
        \bigcup_{s\ne r} Q_{j_s}.
\]

Remove the $d-1$ disjoint intervals $Q_{j_s}$, $s\ne r$, from the cyclic
vertex set $\Zn$. The complement has
\[
        n-(d-1)q=qd+y-(d-1)q=q+y
\]
vertices. Since the intervals $Q_i$, $i\in S_r$, are pairwise intersecting by
Lemma \ref{lem:turan-step}, they must all lie in a single connected component
of this complement; denote that component by $H$. Indeed, two intervals
contained in different connected components of the complement would be
disjoint.

Let $L=|H|$. Since $H$ contains at least one interval $Q_i$ of length $q$, we
have $L\ge q$. On the other hand, $H$ is contained in a complement of total
size $q+y$, so
\[
        L\le q+y<2q.
\]
Write
\[
        L=q+\ell,\qquad 0\le \ell\le y<q.
\]
Cut the cycle open along $H$, and label the vertices of $H$ linearly as
\[
        1,2,\ldots,L.
\]
After this cut, every interval $Q_i\subseteq H$, $i\in S_r$, is an ordinary
linear interval of $q$ consecutive vertices. If its right endpoint is $t$,
then necessarily
\[
        q\le t\le L,
\]
because the interval has length $q$ and must remain inside $H$. Hence the
possible right endpoints lie in the interval
\[
        \{q,q+1,\ldots,L\},
\]
which has length
\[
        L-q+1=\ell+1\le q.
\]

List the right endpoints corresponding to indices in $S_r$ in increasing
linear order:
\[
        t_1<t_2<\cdots<t_m.
\]
Let $i_k\in S_r$ be the original index whose interval $Q_{i_k}$ has right
endpoint $t_k$ in this linear coordinate system.

Fix $k\in\{1,\ldots,m\}$. The cycle $Q_{i_k}$ consists of the backward edge
from $i_k$ to $i_k+1-q$, followed by the forward path
\[
        i_k+1-q\to i_k+2-q\to\cdots\to i_k.
\]
Its weight is therefore
\[
        w_{i_k}=(1-u_{i_k})
        \prod_{a=i_k-q+1}^{i_k-1}u_a,
\]
where the product is taken over the vertices on this forward path.

We now identify which factors in this product can be different from $1$. By
definition, $u_a<1$ exactly when $a\in S$. Thus only indices $a\in S$ matter.

First suppose $h<k$. Since all right endpoints $t_h$ lie in
$\{q,\ldots,L\}$, we have
\[
        t_k-t_h\le L-q=\ell<q.
\]
Therefore
\[
        t_k-q+1\le t_h\le t_k-1,
\]
so the original vertex $i_h$ appears among the indices in the product
defining $w_{i_k}$.

Next suppose $h>k$. Then $t_h>t_k$, so $i_h$ does not appear in the forward
path from $i_k+1-q$ to $i_k$.

Finally, let $a\in S_s$ for some $s\ne r$. Then, by the between-part condition
from Lemma \ref{lem:turan-step},
\[
        Q_a\cap Q_{i_k}=\varnothing.
\]
Since $a\in Q_a$, it follows that $a\notin Q_{i_k}$, and hence $a$ does not
appear in the product defining $w_{i_k}$.

Consequently, the only indices $a\in S$ appearing in the product
\[
        \prod_{a=i_k-q+1}^{i_k-1}u_a
\]
are precisely
\[
        i_1,\ldots,i_{k-1}.
\]
All other factors in the product have $a\notin S$, hence $u_a=1$. Therefore
\[
        w_{i_k}=(1-u_{i_k})\prod_{h<k}u_{i_h}.
\]
Summing over $k=1,\ldots,m$, we obtain the telescoping identity
\[
\begin{aligned}
        \sum_{i\in S_r}w_i
        &=(1-u_{i_1})
          +u_{i_1}(1-u_{i_2})
          +\cdots
          +u_{i_1}\cdots u_{i_{m-1}}(1-u_{i_m})\\
        &=1-u_{i_1}u_{i_2}\cdots u_{i_m}.
\end{aligned}
\]
Since $S_r=\{i_1,\ldots,i_m\}$, this is exactly
\[
        \sum_{i\in S_r}w_i=1-\prod_{i\in S_r}u_i.
\]
This proves the lemma.
\end{proof}

We can now complete the proof of the interior case.

\begin{proposition}\label{prop:interior}
Theorem \ref{thm:main} holds for $0<\alpha<1$.
\end{proposition}

\begin{proof}
Let
\[
        S=S_1\sqcup\cdots\sqcup S_d
\]
be the complete $d$-partite decomposition supplied by Lemma
\ref{lem:turan-step}. Define
\[
        B_r=\{(i,i+1-q):i\in S_r\},\qquad r=1,\ldots,d.
\]
Then
\[
        E(\Gamma)=E_n\cup\bigcup_{r=1}^d B_r.
\]
The distance conditions (P1) and (P2) are exactly the two distance conclusions
of Lemma \ref{lem:turan-step}.

For the product condition, combine Lemmas \ref{lem:turan-step} and
\ref{lem:telescoping}:
\[
        \beta=\sum_{i\in S_r}w_i=1-\prod_{i\in S_r}u_i.
\]
Since $\beta=1-\alpha$, this gives
\[
        \prod_{i\in S_r}u_i=\alpha,\qquad r=1,\ldots,d.
\]
This is (P3). Hence $\Gamma$ satisfies the hypotheses of Kirkland--\v{S}migoc
\cite[Proposition 7.1]{KirklandSmigoc}.
\end{proof}

\section{Endpoint cases}\label{sec:endpoints}

The proof above used $0<\alpha<1$ in two places: first, to ensure that the
realisation lies on the non-degenerate Type III boundary arc; second, to
ensure that $\beta>0$, so that every Tur\'an part has positive total
$q$-cycle weight. The endpoint $\alpha=1$ is nevertheless included in the main
theorem. The endpoint $\alpha=0$ is not.

\begin{proposition}
Theorem \ref{thm:main} holds for $\alpha=1$.
\end{proposition}

\begin{proof}
If $\alpha=1$, then
\[
        f_1(x)=x^n-1.
\]
Let $A$ be a stochastic matrix with characteristic polynomial $x^n-1$. Then
$|\det A|=1$. If $r_i$ denotes the $i$th row of $A$, then $r_i$ is
nonnegative and has entries summing to $1$, so
\[
        \|r_i\|_2\le \|r_i\|_1=1.
\]
Hadamard's inequality gives
\[
        1=|\det A|\le\prod_{i=1}^n\|r_i\|_2\le1.
\]
Thus equality holds throughout. Hence every row has Euclidean norm $1$ and
$\ell^1$-norm $1$, so each row is a standard basis vector. Equality in
Hadamard's inequality also forces the rows to be pairwise orthogonal.
Therefore $A$ is a permutation matrix.

The characteristic polynomial of a permutation matrix is the product of
factors $x^{m_j}-1$, where $m_j$ are the cycle lengths of the permutation.
Since this product is $x^n-1$, the permutation consists of one cycle of length
$n$. After permutation similarity, $A=C_n$ and $E(\Gamma)=E_n$.

Now take $B_1=\cdots=B_d=\varnothing$. Conditions (P1) and (P2) are
vacuous, and (P3) holds with the empty-product convention:
\[
        \prod_{i\in\varnothing}u_i=1=\alpha.
\]
Thus the hypotheses of Kirkland--\v{S}migoc
\cite[Proposition 7.1]{KirklandSmigoc} are satisfied.
\end{proof}

\begin{proposition}
If Conjecture \ref{conj:KS} is read literally with $\alpha=0$ included, it is
false.
\end{proposition}

\begin{proof}
Take
\[
        q=2,
        \qquad d=2,
        \qquad y=1,
        \qquad n=5,
        \qquad \alpha=0.
\]
This is the smallest possible parameter choice in the Type III range:
$q\ge2$, $d\ge2$, and $y\ge1$, hence $n=qd+y\ge5$.
This is a genuine Type III choice: the fractions $2/5$ and $1/2$ are Farey
neighbours, since $|2\cdot2-1\cdot5|=1$, and their denominators are $n=5$
and $q=2$.
Then
\[
        f_0(x)=x(x^2-1)^2.
\]
Consider the stochastic matrix
\[
        A=
        \begin{pmatrix}
        0&1&0&0&0\\
        1&0&0&0&0\\
        0&0&0&1&0\\
        0&0&1&0&0\\
        1&0&0&0&0
        \end{pmatrix}.
\]
After ordering the closed classes first, this matrix is block lower
triangular with two $2$-cycle permutation blocks and one zero diagonal block.
Hence
\[
        \det(xI-A)=x(x^2-1)^2=f_0(x).
\]
Its directed graph has two closed $2$-cycles and one transient vertex feeding
into one of them. In particular, it has no $5$-cycle. But every graph
satisfying Kirkland--\v{S}migoc \cite[Proposition 7.1]{KirklandSmigoc}
contains $E_5$, hence contains a $5$-cycle. Therefore this $A$ is not
permutation similar to any matrix whose directed graph satisfies
Kirkland--\v{S}migoc \cite[Proposition 7.1]{KirklandSmigoc}.
\end{proof}

\begin{remark}
The failure at $\alpha=0$ is not a small-dimensional accident. At this endpoint
\[
        f_0(x)=x^y(x^q-1)^d.
\]
The nonzero roots are merely $q$th roots of unity, each with multiplicity $d$;
they already occur in order $q$. Thus the endpoint no longer gives genuinely
$n$-dimensional boundary data in $\Theta_n\setminus\Theta_{n-1}$, and the
Dmitriev--Dynkin/cycle-length input used in Lemma \ref{lem:two-shift} no
longer forces an $n$-cycle.

More generally, for any admissible $q,d,y$, one can take $d$ closed
$q$-cycle permutation blocks and let the remaining $y$ states be transient,
feeding into one of the closed classes. In block form, after ordering the
closed classes first, this has the shape
\[
        A_0=
        \begin{pmatrix}
        C_q\oplus\cdots\oplus C_q & 0\\
        R & 0_y
        \end{pmatrix},
\]
where $C_q$ is the $q$-cycle permutation matrix and $R$ has nonnegative rows
summing to $1$. Then $A_0$ is stochastic and
\[
        \det(xI-A_0)=x^y(x^q-1)^d=f_0(x),
\]
but its recurrent classes all have period $q$ and the transient states lie on
no directed cycle. In particular, no $n$-cycle is forced. This is exactly the
mechanism exhibited by the $n=5$ counterexample above.
\end{remark}

Combining Proposition \ref{prop:interior} and the $\alpha=1$ proposition
proves Theorem \ref{thm:main}.

\section{Concluding remarks}

We have proved the Kirkland--\v{S}migoc Type III realisation conjecture for
the full nonzero parameter range $0<\alpha\le1$. More precisely, every
stochastic realisation of a genuine Type III reduced Ito polynomial of order
$n$ is, after permutation similarity, one of the realisations described by
Kirkland and \v{S}migoc: its directed graph consists of the full $n$-cycle
together with backward edges arranged into $d$ classes satisfying the
prescribed separation conditions, and the forward-edge weights in each class
have product $\alpha$.

The proof separates the analytic boundary input from the finite realisation
argument. The only use of the Karpelevi\v{c}--Ito theory and of the
Dmitriev--Dynkin theorem is to place every realisation, for $0<\alpha<1$, in
the two-shift normal form
\[
        A=DC_n+(I_n-D)C_n^{\,1-q}.
\]
Once this normal form is obtained, the rest of the proof is purely
combinatorial and algebraic. Coates' formula applied to the coefficients of
$x^{n-q}$ and $x^{n-2q}$ forces equality in a weighted Tur\'an inequality.
This equality is rigid: the $q$-cycles arising from the backward edges must
split into $d$ complete multipartite classes of equal total weight. The
circular-arc telescoping lemma then converts this additive equality of
$q$-cycle weights into the multiplicative condition on the forward-edge
weights required in the Kirkland--\v{S}migoc construction.

Thus the apparent freedom in the non-sparsest Type III realisations is
illusory. Extra backward edges may be present, but they cannot be placed or
weighted arbitrarily. The characteristic polynomial forces them to organise
exactly as in the Kirkland--\v{S}migoc family.

The endpoint $\alpha=1$ is degenerate but harmless: the polynomial is
$x^n-1$, and any stochastic realisation is necessarily an $n$-cycle
permutation matrix. By contrast, the endpoint $\alpha=0$ is genuinely
exceptional. There the polynomial reduces to
\[
        x^y(x^q-1)^d,
\]
and reducible stochastic realisations exist with $d$ closed $q$-cycles and
$y$ transient states. Such realisations have no forced $n$-cycle, so they
cannot satisfy the Kirkland--\v{S}migoc graph conditions. This shows that the
exclusion of $\alpha=0$ is not a technical artefact but a sharp boundary of
the classification.

The result is therefore structural rather than analytic: it classifies
stochastic matrices with the prescribed Type III characteristic polynomial
once the genuine Type III boundary setting is in force, without reproving the
Karpelevi\v{c}--Ito parametrisation or tracking individual root branches.

\section*{Funding}

Vincent Ginis acknowledges support from the Research Foundation -- Flanders
(FWO) under grants No.~G032822N and G0K9322N.

\section*{Declaration of competing interest}

The authors declare that they have no competing interests.

\section*{Data availability}

No data were used for the research described in this article.

\section*{Acknowledgements}

The authors would like to thank Carlo Emerencia for his enthusiastic and
careful proofreading; his remarks made the manuscript substantially better.

\end{document}